\RequirePackage{amsthm}
\documentclass[sn-mathphys-num]{sn-jnl}

\usepackage{multirow}%
\usepackage{amsmath,amssymb,amsfonts}%
\usepackage{amsthm}%
\usepackage{mathrsfs}%
\usepackage[title]{appendix}%
\usepackage{xcolor}%
\usepackage{textcomp}%
\usepackage{manyfoot}%
\usepackage{booktabs}%
\usepackage{algorithm}%
\usepackage{algorithmicx}%
\usepackage{algpseudocode}%
\usepackage{listings}%
\usepackage{anyfontsize}

\theoremstyle{thmstyleone}%
\newtheorem{theorem}{Theorem}[section]
\newtheorem{proposition}{Proposition}%

\theoremstyle{thmstyletwo}%
\newtheorem{remark}{Remark}%

\theoremstyle{thmstylethree}%
\newtheorem{definition}{Definition}%

\usepackage{graphicx}
\graphicspath{ {images/} }
\usepackage{enumitem} 
\usepackage{bbold} 
\usepackage{subfig}
\usepackage{tikz-cd}
\usepackage{hyperref}
\usepackage{color}
\usepackage{soul}

\usepackage{dsfont}
\usetikzlibrary{matrix}
\usepackage{musicography}

\graphicspath{ {images/} }

\newcommand{\Z}{\mathbb{Z}}

\newcommand{\R}{\mathbb{R}}

\newtheorem{lemma}[theorem]{Lemma}

\theoremstyle{definition}
\usepackage{enumitem}
\newlist{steps}{enumerate}{1}
\setlist[steps, 1]{label = Step \arabic*:}

\makeatletter
\newcommand{\colim@}[2]{%
  \vtop{\m@th\ialign{##\cr
    \hfil$#1\operator@font colim$\hfil\cr
    \noalign{\nointerlineskip\kern1.5\ex@}#2\cr
    \noalign{\nointerlineskip\kern-\ex@}\cr}}%
}
\newcommand{\colim}{%
  \mathop{\mathpalette\colim@{\rightarrowfill@\scriptscriptstyle}}\nmlimits@
}
\renewcommand{\varprojlim}{%
  \mathop{\mathpalette\varlim@{\leftarrowfill@\scriptscriptstyle}}\nmlimits@
}
\renewcommand{\varinjlim}{%
  \mathop{\mathpalette\varlim@{\rightarrowfill@\scriptscriptstyle}}\nmlimits@
}
\makeatother

\newcommand{\metric}[2]{\langle{#1},{#2}\rangle}

\newcommand{\Ric}[2]{\text{Ric}(#1,#2)}

\begin{document}

\title[Article Title]{Toric Einstein 4-manifolds with non-negative sectional curvature}


\author*[1]{\fnm{Tianyue} \sur{Liu}}\email{ltianyue@sas.upenn.edu}



\affil*[1]{\orgdiv{Department of Mathematics}, \orgname{University of Pennsylvania}, \orgaddress{\street{209 S 33rd Street}, \city{Philadelphia}, \postcode{19104}, \state{PA}, \country{USA}}}

 

\abstract{
We prove that $T^2$-invariant Einstein metrics with non-negative sectional curvature on a four-manifold are locally symmetric.}



\maketitle

\section{Introduction}
Einstein manifolds are manifolds equipped with a Riemannian metric which solves the Einstein equation
\begin{align}\label{eq:1}
\text{Ric}(g)=\Lambda g
\end{align} We shall focus our attention on closed 1-connected four dimensional manifolds, and we are only concerned with the case $\Lambda>0$, as we will be studying those manifolds with continuous symmetry. In this context very little is known about the solutions of these nonlinear PDEs, except when we add extra assumptions to simplify the situation. In particular, all currently-known examples satisfy at least one of the following two conditions:
\begin{enumerate}
\item The metric is conformally K\"ahler.
\item The metric is invariant under some $T^2$-action.
\end{enumerate} 
In the first case, there is a complete classification: the only Hermitian Einstein metrics are 
the Page metric on $\mathbb{CP}^2\# \overline{\mathbb{CP}}^2$ \cite{Pa78}, the Chen-LeBrun-Weber metric on $\mathbb{CP}^2\# 2\overline{\mathbb{CP}}^2$ \cite{CLW08}, and the K\"ahler-Einstein metric on other del Pezzo surfaces (c.f. \cite{OSS16}\cite{Ti90}). 
A remarkable feature of the non-K\"ahler examples is that they are all invariant under a biholomorphic $T^2$-action, which simplifies the geometry significantly.\par
In the second case much less is known. As mentioned in the survey paper of Anderson \cite{An10}, it would be interesting to know which four-dimensional manifolds admitting $T^2$-actions also have $T^2$-invariant Einstein metrics. Moreover, since on most of these manifolds there exist many different $T^2$-actions, it is also interesting to ask if on some of them there could be multiple Einstein metrics invariant under different $T^2$-actions. Topologically, four-manifolds admitting a $T^2$-action are obtained by taking connected sums of $S^4,\mathbb{CP}^2,\bar{\mathbb{CP}}^2$ and $S^2\times S^2$. It is known that these connected sums all admit metrics of positive Ricci curvature \cite{SY90} and $T^2$-invariant metrics of positive scalar curvature (\cite{Wi15}, Corollary 2.6), but to the author's knowledge, on most of them it is not known if there are $T^2$-invariant metrics with positive Ricci curvature. Nevertheless, for the diagonal $T^2$-actions on connected sums of $S^2\times S^2$ the construction of Sha and Yang in \cite{SY90} yields invariant metrics of positive Ricci curvature. (See Section 4 of \cite{Re24B} for more details.)
As for $T^2$-invariant Einstein metrics with $\Lambda>0$, only six examples are known so far: the round metric on $S^4$, the Fubini-Study metric on $\mathbb{CP}^2$, the product metric on $S^2\times S^2$, the Page metric, the Chen-LeBrun-Weber metric, and the unique K\"ahler Einstein metric on $\mathbb{CP}^2\# 3\overline{\mathbb{CP}}^2$ \cite{Si88}. On the other hand, the best known obstruction to the existence of Einstein metric with $\Lambda>0$ is given by the Hitchin-Thorpe inequality \cite{Hi74} 
\[2\chi\pm3\tau> 0
\]
where $\chi$ is the Euler characteristic and $\tau$ is the signature. \footnote{For various improvements of the Hitchin-Thorpe inequality, see e.g. \cite{Gr82},\cite{Le98},\cite{Sa98}} If we impose more restrictions on the Einstein metric, then there are further results: In \cite{GL98}, Gursky and LeBrun showed that there is no Einstein metric with non-negative sectional curvature other than the Fubini-Study metric if the intersection form is positive definite, regardless of whether the metric admits isometries.\par
The main result of this paper is the following: 
\begin{theorem}\label{1.1}
Assume $M^4$ is closed and simply connected. If $M^4$ has a $T^2$-invariant Einstein metric $g$ with non-negative sectional curvature, then $(M^4,g)$ is isometric to the product $S^2\times S^2$ of two spheres of the same radii, the round sphere $S^4$ or a constant multiple of the Fubini-Study metric on $\mathbb{CP}^2$.
\end{theorem}
In particular, any closed toric Einstein 4-manifold of non-negative sectional curvature must be locally symmetric.\par

We shall now sketch the proof of our main result. The strategy is to use the non-negativity of sectional curvature as well as the existence of symmetry to pin down the diffeomorphism type first, then investigate each diffeomorphism type case-by-case. \par    
    According to \cite{HK89}, \cite{GS11} and \cite{GW14}, the non-negativity of sectional curvature along with the existence of continuous symmetry implies that $M^4$ is diffeomorphic to one of $S^4,\mathbb{CP}^2,S^2\times S^2$, $\mathbb{CP}^2\#\overline{\mathbb{CP}}^2$, or  $\mathbb{CP}^2\#\mathbb{CP}^2$. By \cite{GL98}, there is no non-negatively curved Einstein metric on $\mathbb{CP}^2\#\mathbb{CP}^2$, and on $\mathbb{CP}^2$ there is only the Fubini-Study metric. Hence assuming $M^4$ is as stated in \ref{1.1}, then one of the following holds:\par
    \begin{enumerate}
        \item $\chi(M^4)=2$, and $M^4$ is diffeomorphic to $S^4$.
        \item $\chi(M^4)=3$, and $(M^4,g)$ is isometric to the Fubini-Study metric on $\mathbb{CP}^2$.
        \item $\chi(M^4)=4$, and $M^4$ is diffeomorphic to either $S^2\times S^2$ or $\mathbb{CP}^2\#\overline{\mathbb{CP}}^2$.   
    \end{enumerate}
We shall first show in section \ref{5} that in the $\chi(M)=4$ case $(M^4,g)$ is isometric to $S^2\times S^2$ with the product Einstein metric,
then proceed in section \ref{6} to show that in the $\chi(M)=2$ case $(M^4,g)$ is isometric to the round sphere.
\par
On $S^2\times S^2$ or $\mathbb{CP}^2\#\bar{\mathbb{CP}}^2$ there are infinitely many torus actions up to equivariant diffeomorphism. These actions arise as automorphisms of different Hirzebruch surfaces. A $T^2$-invariant metric on these 4-manifolds induces a quotient metric on the orbit space $\Sigma=M^4/T^2$, making it a curved quadrilateral with geodesic edges and right angle corners. To classify non-negatively curved $T^2$-invariant Einstein metrics on them, we first show the following general statement, which holds regardless of whether non-negative sectional curvature is assumed:
\begin{theorem}
    An isometric $T^2$-action on a closed and simply connected Einstein 4-manifold $M^4$ is geometrically polar---i.e. there exists a section $\sigma:\Sigma \to M^4$ everywhere perpendicular to the orbits. Moreover, given an action on $M^4$ and a metric $g_\Sigma $ on the orbit space, there is at most one invariant Einstein metric on $M^4$ with induced metric on the orbit space isometric to $g_\Sigma$.
\end{theorem}  
 The proof of the first statement can be found in section \ref{3}, and in section \ref{4} we list the curvature formulae that lead to a proof of the second statement in section \ref{5}. We note that the first result has also been found by \cite{Da06}. Since the section $\sigma$ provides a totally geodesic embedding of $\Sigma$ into $M^4$, the assumption of non-negative sectional curvature forces the orbit space to be flat. \footnote{We later learned that the proof of lemma 2.2 in \cite{GS11} implicitly implies that the quotient here is a flat rectangle. Nevertheless, we still need the metric to be geometric polar to solve the Einstein equations.} We then investigate the sizes of the singular strata, and show the orbit space actually must be a square in Lemma \ref{5.1}. Then in Proposition \ref{3} we solve the Einstein equation explicitly to deduce that only the product action on $S^2\times S^2$ could preserve Einstein metrics with square orbit space. We then conclude from the solution to Einstein equation that the unique possibility in the case of $\chi(M)=4$ is the product metric on $S^2\times S^2$. \par
On $S^4$, there is only one torus action up to equivariant diffeomorphism, which is a diagonal action (see definition \ref{def:1}). In section \ref{6} we show that if an Einstein manifold (with no other curvature assumption) is invariant under a diagonal action, then it must be a diagonal metric, i.e. there are two Killing fields perpendicular to each other everywhere. Then we show in Theorem \ref{6.2} for diagonal metrics, non-negative sectional curvature implies non-negative isotropic curvature, hence a result of Brendle in \cite{Br10} can be applied to show any such Einstein metric must be locally symmetric, and thus isometric to a round sphere $S^4$.\par
We shall extend our classification of simply connected examples into a classification of orientable examples in section \ref{7}. Besides the classification results, we also show in section \ref{8} that geometrically polar and diagonal actions are preserved under the Ricci flow. 
\par
The author would like to thank his thesis advisor Wolfgang Ziller for his constant support and encouragement. He would also like to thank Dennis DeTurck, Davi Maximo, Renato Bettiol, McFeely Jackson Goodman, Sammy Sbiti, and Yingdi Qin for many helpful conversations. Finally, he would like to thank Michael Wiemeler for pointing out the reference \cite{Wi15}, thank Philipp Reiser for pointing out the reference \cite{Re24B}, and thank Thomas Murphy for pointing out that in the unpublished chapter 4 of Dammerman's thesis \cite{Da06} he has obtained our theorem \ref{3.2} as well as the formulae in section \ref{4}, Finally, he would like to thank the anonymous reviewer for their invaluable suggestions. This work is part of the author's PhD thesis.\par

\section{\texorpdfstring{Smoothness conditions for $T^2$-}{Smoothness conditions for T2-}invariant metrics}\label{2}
The purpose of this section is to review Orlik and Raymond's description of $T^2$ invarant actions, and to determine the smoothness conditions for $T^2$-invariant metrics. In particular, in lemma \ref{2.1} we discuss the smoothness conditions near the codimension-two strata, and in lemma \ref{2.2} we discuss the smoothness conditions near the fixed points.\par
Orlik and Raymond \cite{OR70} classified closed simply connected toric 4-manifolds up to equivariant diffeomorphism. The torus actions endow these manifolds with a natural stratification based on the isotropy groups. The principal stratum is an open dense subset consisting of points where the isotropy groups are trivial. Topologically, it is diffeomorphic to $T^2\times D^2$. There are two kinds of singular strata contained in the closure of the principal stratum: The codimension-two strata, which form an open dense subset of the singular strata, consists of many connected components where the isotropy groups are circles inside $T^2$ of some slope. We shall denote the circle of slope $(p,q)$ as $S(p,q)$, i.e. $S(p,q)=\{(e^{ip\theta},e^{iq\theta})|\theta\in[0,2\pi)\}\subseteq T^2$. We shall always assume $\text{gcd}(p,q)=1$. Different connected components of codimension-two strata are separated by the codimension-four strata, which are fixed points of the action. Such a stratification of the 4-manifold induces a stratification of the orbit space, which is homeomorphic to a two dimensional polygon: The principal stratum descends to the interior of the polygon, the codimension-two strata descends to the edges, and the codimension-four strata descends to the vertices. Conversely, if one starts with a two dimensional polygon $\Sigma$ with each edge $e_i$ labeled by an integer-valued vector $(a_i,b_i)$ such that $\text{gcd}(a_i,b_i)=1$, then as long as at all vertices $v_i=\overline{e_i}\cap \overline{e_{i+1}}$ one has $\det\begin{bmatrix}
    a_i&a_{i+1}\\
    b_i&b_{i+1}
\end{bmatrix}=\pm1$, one can reconstruct a toric 4-manifold $M$ with isotropy groups of codimension-two strata having the corresponding slopes. Orlik and Raymond \cite{OR70} also showed that these torus actions are all topologically polar, i.e. there is always a smooth section $\sigma:\Sigma\to M$, by first showing that the action is topologically polar near the singular strata, and then applying obstruction theory to extend the section to the principal stratum. If $\Z_2\times \Z_2\subseteq T^2$ is the subgroup of two-torsion points, then one can use this subgroup to reflect the image $\sigma(\Sigma)$ and obtain a closed embedded surface $\tilde{\Sigma}$ in $M$. If $g$ is a metric on $M$ invariant under the torus action, then by the slice neighborhood theorem the induced metric $g^\Sigma$ on the orbit space $\Sigma$ must be isometric to a (curved) polygon with all angles being right angles, and all edges being geodesics. However, it may not be isometric to the metric on $\Sigma$ induced by the topological section, as the topological section may not be everywhere perpendicular to the orbits.\par
In the rest part of this section we shall compute the Taylor expansions of smooth $T^2$-invariant metrics on $M$ near the singular strata described above. \par
We first deal with the codimension-two strata. To begin with, we describe the torus action on a slice neighborhood. Let $p\in M$ be a point in one of the codimension-two stratum $Z$, and $O:=T^2.p$ be its orbit. Without loss of generality, we can assume that the isotropy group of $p$ is $S(1,0)$. We shall use the notation $\nu^{\leq \epsilon}(A/B)$ to denote the set of vectors of length less than or equal to $\epsilon$ in the normal bundle of $B$ considered as a submanifold of $A$. By the slice neighborhood theorem, for small enough $\epsilon$ there is a neighborhood of $O$ in $Z$, denoted as $V$, such that the exponential map gives a diffeomorphism from $\nu^{\leq \epsilon}(Z/O)$ to $V$. There is also a neighborhood of $V$ in $M$, denoted as $U$, such that the exponential map gives a diffeomorphism from $\nu^{\leq \epsilon}(M/V)$ to $U$. Since both of the normal bundles are trivial, we can identify $\nu^{\leq \epsilon}(Z/O)$ with $S^1\times D^1$ and $\nu^{\leq \epsilon}(M/V)$ with $D^2\times S^1\times D^1$. Through Orlik and Raymond's discussion of isotropic representations (c.f. \cite{OR70}, p534), if we let $T^2$ act on $D^2\times S^1\times D^1$ via rotation, i.e. for $(e^{i\theta_1},e^{i\theta_2})\in T^2$ and $ (re^{i\varphi},\theta,t)\in D^2\times S^1\times D^1$, the action is given by $(e^{i\theta_1},e^{i\theta_2}).(re^{i\varphi},\theta,t)= (re^{i(\varphi+\theta_1)},\theta+\theta_2,t)$, then the exponential map gives an equivariant diffeomorphism from $D^2\times S^1\times D^1$ to $U$. We can use this diffeomorphism to pull the metric on $U$ to a metric on $D^2\times S^1\times D^1$.\par
We now write down this metric explicitly using the coordinates $(r,\varphi,t,\theta)$ on $(D^2-\{0\})\times S^1\times D^1$. Since $(r,t)$ gives the Fermi coordinates over the orbit space, the metric there can be written as $g^\sigma:=dr^2+\mu^2(r,t)dt^2$. The metric along orbits is given by $g^\Omega:=a^2d\varphi^2+2bd\varphi d\theta+d^2d\theta^2$.
Let $X=\partial_r+c_{11}(r,t)\partial_\varphi+c_{12}(r,t)\partial_\theta$ and $Y=\partial_t+c_{21}(r,t)\partial_\varphi+c_{22}(r,t)\partial_\theta$ be the horizontal lifts of the vector fields $\partial_r,\partial_t$, and put the coefficients into a matrix $C:=\begin{bmatrix}
        c_{11}&c_{12}\\
        c_{21}&c_{22}
    \end{bmatrix}$. The metric on $D^2\times S^1\times D^1$ is then given by the following matrix with coefficients depending on $(r,t)$:
    \begin{align}\label{eq:2}
    \begin{bmatrix}
        g^\Sigma+Cg^\Omega C^T&-Cg^\Omega\\
        -g^\Omega C^T&g^\Omega
    \end{bmatrix}
\end{align}
\begin{lemma}\label{2.1}
 Let $R^{even}$ denote the space of real valued functions on $[0,\epsilon)\times (-\epsilon,\epsilon)$ that extends smoothly to $(-\epsilon,\epsilon)^2$ as an even function in the first variable, then the coefficients $(\mu, a,b,d, C)$ define a smooth metric on $D^2\times S^1\times D^1$ if and only if the following conditions are satisfied:
    \begin{enumerate}
        \item $\mu^2, r^{-1}c_{11}, r^{-1}c_{12},c_{21},c_{22}, a^2,b,d^2\in R^{even}$
        \item For all values of $t$, we have $a^2(0,t)=0$, $\lim_{r\to 0^+}\frac{1}{r^2}a^2(r,t)=1$, $b(0,t)=0$, $\mu(0,t)=1$, $d^2(0,t)> 0$.
    \end{enumerate}
\end{lemma}
\begin{proof}
    We need to change to a coordinate chart that covers $\{0\}\times S^1\times D^1$. Let the new coordinate chart be given by $(x=r\cos\varphi,y=r\sin\varphi,t=t,\theta=\theta)$, then the components of the metric in the new chart is given in the following table:

    \begin{align*}
    g^{xx}&=\frac{x^2}{r^2}(1+(c_{11}^2a^2+2c_{11}c_{12}b+c_{12}^2d^2)) 
    -\frac{2}{r^3}xy(a^2c_{11}+bc_{12})
    +\frac{y^2}{r^4}a^2\\
    g^{yy}&=\frac{y^2}{r^2}(1+(c_{11}^2a^2+2c_{11}c_{12}b+c_{12}^2d^2))    
    +\frac{2}{r^3}xy(a^2c_{11}+bc_{12})
    +\frac{x^2}{r^4}a^2\\
    g^{xy}&=2[\frac{xy}{r^2}(1+(c_{11}^2a^2+2c_{11}c_{12}b+c_{12}^2d^2))
    -\frac{1}{r^3}(x^2-y^2)(a^2c_{11}+bc_{12})
    -\frac{xy}{r^4}a^2]\\
    g^{xt}&=2\frac{x}{r}(c_{11}c_{21}a^2+(c_{12}c_{21}+c_{11}c_{22})b+c_{12}c_{22}d^2)+\frac{4}{r^2}y(a^2c_{21}+bc_{22})\\
    g^{yt}&=2\frac{y}{r}(c_{11}c_{21}a^2+(c_{12}c_{21}+c_{11}c_{22})b+c_{12}c_{22}d^2)-\frac{4}{r^2}x(a^2c_{21}+bc_{22})\\
    g^{x\theta}&=-\frac{2b}{r^2}y-\frac{2}{r}x(bc_{11}+d^2c_{12})\\
    g^{y\theta}&=\frac{2b}{r^2}x-\frac{2}{r}y(bc_{11}+d^2c_{12})\\
    g^{tt}&=\mu^2+(c_{21}^2a^2+2c_{21}c_{22}b+c_{22}^2d^2)\\
    g^{t\theta}&=-2(bc_{21}+d^2c_{22})\\
    g^{\theta\theta}&=d^2
\end{align*}
    We recall that a radial function $f(|z_1|)$ on $D^2$ is smooth if and only if $f$ extends to a smooth even function over a neighborhood of $0\in \R$.  If $(\mu, a,b,d, C)$ satisfies the conditions mentioned in the lemma, then one can verify the smoothness of the functions in the table directly. Conversely, we assume the function in the table are smooth, and we need to deduce the conditions for $(\mu, a,b,d, C)$ in the lemma.\par
    By restricting these components of the metric tensor to the lines $\gamma_{(t,\theta)}:s\mapsto (s,0,t,\theta)$ inside the hypersurface where $y=0$, we can see that $(c_{11}^2a^2+2c_{11}c_{12}b+c_{12}^2d^2)$, $\frac{a^2}{r^2}$, $\frac{1}{r}(a^2c_{11}+bc_{12})$, $(\mu^2+(c_{21}^2a^2+2c_{21}c_{22}b+c_{22}^2d^2))$, $(bc_{21}+d^2c_{22})$, $d^2$ have smooth even extensions, and $(c_{11}c_{21}a^2+(c_{12}c_{21}+c_{11}c_{22})b+c_{12}c_{22}d^2)$, $\frac{1}{r}(a^2c_{21}+bc_{22})$, $\frac{b}{r}$, $(bc_{11}+d^2c_{12})$ have smooth odd extensions. In particular, $a^2,b,d^2$ all admit smooth even extensions, and $a^2(0,t)=b(0,t)=0$. Since $a^2$ is the length of a Killing field and a Killing field cannot vanish up to second order, we see $\lim_{r\to 0^+}\frac{1}{r^2}a^2(r,t)>0$. Since $d^2$ is the length of a Killing field that does not vanish at $r=0$, we see $d^2(0,t)>0$.\par
    We now show that $r^{-1}c_{11}, r^{-1}c_{12}\in R^{even}$. Since $c_{11}a^2+c_{12}b$, $c_{11}b+c_{12}d^2$ both admit smooth extensions to odd functions in $r$, after multiplying them by smooth even functions we see $(a^2d^2-b^2)c_{11}, (a^2d^2-b^2)c_{12}$ admit smooth odd extensions. From the order of vanishing for $a^2,b,d^2$ we see this implies $r^2c_{11},r^2c_{12}\in rR^{even}$. We still need to show $\lim_{r\to 0^+}rc_{11}(r,t)=\lim_{r\to 0^+}rc_{12}(r,t)= 0$. If $\lim_{r\to 0^+}rc_{12}(r,t)\neq 0$, then since $c_{11}b\in rR^{even}$, we see $c_{11}b+c_{12}d^2$ cannot be bounded at $r=0$. We thus get $r^{-1}c_{12}\in R^{even}$. If $\lim_{r\to 0^+}rc_{11}(r,t)\neq 0$, then $\frac{xy}{r^3}(a^2c_{11}+bc_{12})$ cannot be smooth, which prevents $g^{xx}$ to be smooth as it cannot be cancelled by other terms in the expression for $g^{xx}$. We thus get $r^{-1}c_{11}\in R^{even}$.\par
    The proof for $c_{21},c_{22}\in R^{even}$ is similar: Since $c_{21}a^2+c_{22}b$, $c_{21}b+c_{22}d^2$ both admit smooth extensions to even functions, we get $(a^2d^2-b^2)c_{21}, (a^2d^2-b^2)c_{22}\in R^{even}$, which implies that $r^2c_{21},r^2c_{22}\in R^{even}$. If $\lim_{r\to 0^+}r^2c_{22}(r,t)\neq 0$, then since $c_{21}b\in R^{even}$, we see $c_{21}b+c_{22}d^2$ cannot be bounded at $r=0$, so $c_{22}\in R^{even}$. If $\lim_{r\to 0^+}r^2c_{12}(r,t)\neq 0$, then $\lim_{r\to 0^+}(a^2c_{21}+bc_{22})\neq 0$, which would contradict $(a^2c_{21}+bc_{22})\in rR^{even}$.\par
    We thus have $c_{11}^2a^2+2c_{11}c_{12}b+c_{12}^2d^2\in r^2R^{even}$. Combining with the assumption that $g^{xx}=\frac{x^2}{r^2}(1+(c_{11}^2a^2+2c_{11}c_{12}b+c_{12}^2d^2)) 
    -\frac{2}{r^3}xy(a^2c_{11}+bc_{12})
    +\frac{y^2}{r^4}a^2$ is smooth, we see $\frac{x^2}{r^2}+\frac{y^2}{r^2}a^2$ must be smooth, which implies that $\lim_{r\to 0^+}\frac{1}{r^2}a^2(r,t)=1$. From the smoothness of $g^{tt}=\mu^2+(c_{21}^2a^2+2c_{21}c_{22}b+c_{22}^2d^2)$ we get $\mu^2\in R^{even}$, and $\mu(0,t)=1$ follows from the assumption that  $(r,t)$ gives the Fermi coordinates on the orbit space.
\end{proof}
We now deal with the fixed points. As before we start by describing the torus action on a slice neighborhood. Let $p$ be a fixed point, then up to a reparametrization of the torus we can assume the isotropy subgroups of the two adjacent codimension-two strata $Z$, $Z'$ are $S(1,0)$, $S(0,1)$ respectively. Let the quotients of $Z,Z'$ be denoted as $e,e'\subseteq \Sigma$. The slice neighborhood theorem then guarantees that there is a small enough neighborhood $U'\ni p$ and an equivariant diffeomorphism $\psi:D^2\times D^2\to U'$, where the torus acts on $D^2\times D^2\subseteq T_pM$ as rotations: 
 for $(e^{i\theta_1},e^{i\theta_2})\in T^2$ and$(re^{i\varphi},te^{i\theta})\in D^2\times D^2$, $(e^{i\theta_1},e^{i\theta_2}).(re^{i\varphi},te^{i\theta})= (re^{i(\varphi+\theta_1)},te^{i\theta+\theta_2})$. We can assume $\psi$ takes the subset where $r=0$ to $\bar{Z}$ and the subset where $t=0$ to $\bar{Z}'$. Introducing functions $(\mu,a^2,b,d^2,C)$ in $(r,t)$ as before, the pullback metric on $(D^2-\{0\})\times (D^2-\{0\})$ will take the same form as \ref{eq:2}.\par
\begin{lemma}\label{2.2}
 Let $R^{even}$ denote the space of real valued functions on $[0,\epsilon)\times [0,\epsilon)$ that extends smoothly to $(-\epsilon,\epsilon)^2$ as an even function, i.e. a function invariant under the $\Z_2\times \Z_2$-action flipping the coordinates, then the coefficients $(\mu, a,b,d, C)$ define a smooth metric on $D^2\times S^1\times D^1$ if and only if the following conditions are satisfied:
    \begin{enumerate}
        \item $\mu^2, r^{-1}c_{11}, r^{-1}c_{12},t^{-1}c_{21},t^{-1}c_{22}\in R^{even}$
        \item $a^2\in r^2+r^4R^{even}, d^2\in t^2\mu^2(r,0)+t^4R^{even}, b\in r^2t^2R^{even}$
        \item $\mu(0,t)=1$ for all t.
    \end{enumerate}
\end{lemma}
\begin{proof}
    We need to change to a coordinate chart that covers the entire $D^2\times D^2$. Let the new coordinate chart be given by $(x=r\cos\varphi,y=r\sin\varphi,z=t\cos\theta,w=t\sin\theta)$, then the components of the metric in the new chart is given in the following table:
    \begin{align*}
    g^{xx}&=\frac{x^2}{r^2}(1+(c_{11}^2a^2+2c_{11}c_{12}b+c_{12}^2d^2)) 
    -\frac{2}{r^3}xy(a^2c_{11}+bc_{12})
    +\frac{y^2}{r^4}a^2\\
    g^{yy}&=\frac{y^2}{r^2}(1+(c_{11}^2a^2+2c_{11}c_{12}b+c_{12}^2d^2))    
    +\frac{2}{r^3}xy(a^2c_{11}+bc_{12})
    +\frac{x^2}{r^4}a^2\\
    g^{xy}&=2[\frac{xy}{r^2}(1+(c_{11}^2a^2+2c_{11}c_{12}b+c_{12}^2d^2))
    -\frac{1}{r^3}(x^2-y^2)(a^2c_{11}+bc_{12})
    -\frac{xy}{r^4}a^2]\\
    g^{xz}&=\frac{1}{t}z[2\frac{x}{r}(c_{11}c_{21}a^2+(c_{12}c_{21}+c_{11}c_{22})b+c_{12}c_{22}d^2)+\frac{4}{r^2}y(a^2c_{21}+bc_{22})]\\
    &-\frac{1}{t^2}w[-\frac{2b}{r^2}y-\frac{2}{r}x(bc_{11}+d^2c_{12})]\\
    g^{yz}&=\frac{1}{t}w[2\frac{y}{r}(c_{11}c_{21}a^2+(c_{12}c_{21}+c_{11}c_{22})b+c_{12}c_{22}d^2)-\frac{4}{r^2}x(a^2c_{21}+bc_{22})]\\
    &-\frac{1}{t^2}z[\frac{2b}{r^2}x-\frac{2}{r}y(bc_{11}+d^2c_{12})]\\
    g^{xw}&=\frac{1}{t}z[2\frac{x}{r}(c_{11}c_{21}a^2+(c_{12}c_{21}+c_{11}c_{22})b+c_{12}c_{22}d^2)+\frac{4}{r^2}y(a^2c_{21}+bc_{22})]\\
    &+\frac{1}{t^2}z[-\frac{2b}{r^2}y-\frac{2}{r}x(bc_{11}+d^2c_{12})]\\
    g^{yw}&=\frac{1}{t}w[2\frac{y}{r}(c_{11}c_{21}a^2+(c_{12}c_{21}+c_{11}c_{22})b+c_{12}c_{22}d^2)-\frac{4}{r^2}x(a^2c_{21}+bc_{22})]\\
    &+\frac{1}{t^2}z[\frac{2b}{r^2}x-\frac{2}{r}y(bc_{11}+d^2c_{12})\\
    g^{zz}&=\frac{1}{t^2}z^2(\mu^2+(c_{21}^2a^2+2c_{21}c_{22}b+c_{22}^2d^2))+\frac{1}{t^4}w^2d^2+\frac{2}{t^3}zw(bc_{21}+d^2c_{22})\\
    g^{ww}&=\frac{1}{t^2}w^2(\mu^2+(c_{21}^2a^2+2c_{21}c_{22}b+c_{22}^2d^2))+\frac{1}{t^4}w^2d^2-\frac{2}{t^3}zw(bc_{21}+d^2c_{22})\\
    g^{zw}&=\frac{2}{t^2}zw(\mu^2+(c_{21}^2a^2+2c_{21}c_{22}b+c_{22}^2d^2))+\frac{2}{t^4}zwd^2+\frac{2}{t^3}(z^2-w^2)(bc_{21}+d^2c_{22})
\end{align*}
     We recall that a radial function $f(|z_1|,|z_2|)$ on $D^2\times D^2$ is smooth if and only if $f$ extends to a smooth function over a neighborhood of $0\in \R^2$ invariant under the reflections in both coordinates.  If $(\mu, a,b,d, C)$ satisfies the conditions mentioned in the lemma, then one can verify the smoothness of the functions in the table directly. Conversely, we assume the function in the table are smooth, and we need to deduce the conditions for $(\mu, a,b,d, C)$ in the lemma.\par
    By restricting these components of the metric tensor to the embedded disks where $y=w=0$, we can see that $(c_{11}^2a^2+2c_{11}c_{12}b+c_{12}^2d^2)$, $\frac{a^2}{r^2}$, $\frac{1}{r}(a^2c_{11}+bc_{12})$, $(\mu^2+(c_{21}^2a^2+2c_{21}c_{22}b+c_{22}^2d^2))$, $\frac{1}{t}(bc_{21}+d^2c_{22})$, $\frac{d^2}{t^2}$ are functions inside $R^{even}$, and $(c_{11}c_{21}a^2+(c_{12}c_{21}+c_{11}c_{22})b+c_{12}c_{22}d^2)$, $\frac{1}{r}(a^2c_{21}+bc_{22})$, $\frac{b}{rt}$, $\frac{1}{t}(bc_{11}+d^2c_{12})$ are functions inside $rtR^{even}$. In particular, we have $a^2\in r^2R^{even},b\in r^2t^2R^{even},d^2\in t^2R^{even}$. \par
    We then show that $r^{-1}c_{11}, r^{-1}c_{12}\in R^{even}$, and the argument for $t^{-1}c_{21}, t^{-1}c_{22}\in R^{even}$ will be exactly the same. The argument will be parallel to the argument for codimension-two strata. Since $c_{11}a^2+c_{12}b$, $c_{11}b+c_{12}d^2$ are both in $rR^{even}$, after multiplying them by smooth even functions we see $(a^2d^2-b^2)c_{11}, (a^2d^2-b^2)c_{12}$ are in $rR^{even}$ too. From the order of vanishing for $a^2,b,d^2$ we see this implies $r^2c_{11},r^2c_{12}\in rR^{even}$. If $\lim_{r\to 0^+}rc_{12}(r,t)\neq 0$, then since $c_{11}b\in rR^{even}$, we see $c_{11}b+c_{12}d^2$ cannot be bounded at $r=0$. We thus get $r^{-1}c_{12}\in R^{even}$. If $\lim_{r\to 0^+}rc_{11}(r,t)\neq 0$, then once again $\frac{xy}{r^3}(a^2c_{11}+bc_{12})$ cannot be smooth, contradicting the smoothness of $g^{xx}$, and we thus get $r^{-1}c_{11}\in R^{even}$.\par
    We then have $c_{11}^2a^2+2c_{11}c_{12}b+c_{12}^2d^2\in r^2R^{even}$. Combining with the assumption that $g^{xx}=\frac{x^2}{r^2}(1+(c_{11}^2a^2+2c_{11}c_{12}b+c_{12}^2d^2)) 
    -\frac{2}{r^3}xy(a^2c_{11}+bc_{12})
    +\frac{y^2}{r^4}a^2$ is smooth, we see $\frac{x^2}{r^2}+\frac{y^2}{r^2}a^2$ must be smooth, which implies that $a^2\in r^2+r^4R^{even}$. Similarly from the smoothness of $g^{zz}=\frac{1}{t^2}z^2(\mu^2+(c_{21}^2a^2+2c_{21}c_{22}b+c_{22}^2d^2))+\frac{1}{t^4}w^2d^2+\frac{2}{t^3}zw(bc_{21}+d^2c_{22})$ we get $d^2\in t^2\mu^2(r,0)+ t^4R^{even}$, and $\mu(0,t)=1$ follows from the usage of Fermi coordinates.
\end{proof}

\section{Integrability of the horizontal distribution}\label{3}

\begin{definition}
    Let $T^2$ act on the Riemannian manifold $(M,g)$ by isometry. The action is geometrically polar\footnote{It is referred to as surface orthogonal in \cite{Da06}} if it satisfies one of the following equivalent conditions: 
    \begin{enumerate}
        \item there is a section $\sigma:\Sigma\to M$ everywhere perpendicular to the orbits.
        \item The two dimensional distribution (which we refer to as the horizontal distribution) perpendicular to the $T^2$ orbits is integrable.
    \end{enumerate} 
\end{definition}
The Einstein equations impose strong constraints on the geometry of $M$. In particular, for $X$ a vector perpendicular to the $T^2$ orbits (which we refer to as a horizontal vector) and $V$ an action field, we must have $\Ric{X}{V}=0$. This allows us to obtain the following result, which was also proven in \cite{Da06}:
\begin{theorem}\label{3.2}
If $T^2$ acts by isometries on a simply connected closed Einstein 4-manifold, then the action is geometrically polar.
\end{theorem}
\begin{proof}
We first compute the Ricci curvature for a $T^2$-invariant metric and observe that $\Ric{X}{V}=0$ implies that the O'Neill's A-tensor (c.f. \cite{On66}) over the interior of orbit space is constant. Then we use the smoothness conditions to show the O'Neill's A-tensor extends to a continuous function over the entire orbit space that vanishes along the image of singular strata. This implies that the O'Neill's A-tensor vanishes everywhere, so the horizontal distribution is integrable. \par
Choose a coordinate chart on the orbit space polygon $\Sigma$ so that locally the metric $g_\Sigma$ on $\Sigma$ is given by $ds^2=dx^2+\mu^2dy^2$. We can lift $\partial_x,\partial_y$ to horizontal vector fields $X,Y$ on the principal stratum of $M$. Let $V_1,V_2$ denote two linearly independent Killing fields on $M$. We observe that $[X,Y]$ is a vertical vector and the action is polar iff $[X,Y]=0$. Since $X,Y$ are $T^2$ equivariant, we have $[X,V_1]=[X,V_2]=[Y,V_1]=[Y,V_2]=0$. Since $T^2$ is abelian, we also have $[V_1,V_2]=0$. \par
With the notation set up, a straightforward computation shows that the Ricci curvature is given by: 
    \begin{align*}
    \Ric{Y}{V_1}&=\frac{\mu}{2\sqrt{\det(g^\Omega_{ij})}}\partial_x(\frac{\sqrt{\det(g^\Omega_{ij})}}{\mu}\metric{[X,Y]}{V_1})\\
\Ric{X}{V_1}&=-\frac{1}{2\mu\sqrt{\det(g^\Omega_{ij})}}\partial_y(\frac{\sqrt{\det(g^\Omega_{ij})}}{\mu}\metric{[X,Y]}{V_1})\\
\Ric{Y}{V_2}&=\frac{\mu}{2\sqrt{\det(g^\Omega_{ij})}}\partial_x(\frac{\sqrt{\det(g^\Omega_{ij})}}{\mu}\metric{[X,Y]}{V_2})\\
\Ric{X}{V_2}&=-\frac{1}{2\mu\sqrt{\det(g^\Omega_{ij})}}\partial_y(\frac{\sqrt{\det(g^\Omega_{ij})}}{\mu}\metric{[X,Y]}{V_2})
\end{align*}
We observe these formula can be written in a more compact form. Let $J$ denote the complex structure on $\Sigma$ induced by the metric, and make the following definitions:
\begin{align*}
\alpha:=\sqrt{\det(g^\Omega_{ij})}\metric{[X,\frac{1}{\mu}Y]}{V_1}\\
\beta:=\sqrt{\det(g^\Omega_{ij})}\metric{[X,\frac{1}{\mu}Y]}{V_2}
\end{align*}
The Ricci curvature can then be expressed as follows:
\begin{align*}
    \Ric{Y}{V_1}&=-\frac{1}{2\sqrt{\det(g^\Omega_{ij})}}\metric{d\alpha}{JY}&
\Ric{X}{V_1}&=-\frac{1}{2\sqrt{\det(g^\Omega_{ij})}}\metric{d\alpha}{JX}\\
\Ric{Y}{V_2}&=-\frac{1}{2\sqrt{\det(g^\Omega_{ij})}}\metric{d\beta}{JY}&
\Ric{X}{V_2}&=-\frac{1}{2\sqrt{\det(g^\Omega_{ij})}}\metric{d\beta}{JX}
\end{align*}

Since the metric is Einstein, we have $\begin{bmatrix}
\Ric{X}{V_1}&\Ric{X}{V_2}\\
\Ric{Y}{V_1}&\Ric{Y}{V_2}
\end{bmatrix}=0$. Thus $\alpha,\beta$ are constants, as the Einstein equations imply that their gradients vanish over the principal stratum:
\begin{align}\label{eq:3}
\nabla_\Sigma\alpha=
\nabla_\Sigma\beta=0
\end{align}
\par
On the other hand, in the slice neighborhoods described in the previous section, let $X,Y$ be the horizontal lifts of $\partial_r,\partial_t$ and $V_1=\partial_\varphi, V_2=\partial_\theta$, then $[X,Y]=(\partial_rc_{21}-\partial_tc_{11})\partial_\varphi+(\partial_rc_{22}-\partial_tc_{12})\partial_\theta$. Since $c_{ij}$ have all derivatives bounded, we then observe that $\metric{[X,Y]}{V_1}$, $\metric{[X,Y]}{V_2}$ are bounded near the singular strata, therefore the constant functions 
$\alpha,\beta$ over the principal stratum extends to continuous functions that vanish on the singular strata, so they vanish in the slice neighborhood. We then have $[X,Y]=0$ in a tubular neighborhood of the union of all singular strata, which implies that $[X,Y]=0$ everywhere. 
\end{proof}
\begin{remark}
The quantities $\alpha,\beta$ are called the twist scalars in Proposition 4.12 in \cite{Da06}, and seem to be well-known to physicists (c.f. Chapter 19 of \cite{SKMHH03}). Dammerman defined them in a slightly different way: If $\omega_1$, $\omega_2$ are the one forms dual to the Killing fields $V_1,V_2$, then $\alpha=*(\omega_2\wedge\omega_1\wedge d\omega_1)$ and $\beta=*(\omega_1\wedge\omega_2\wedge d\omega_2)$. Using his expression, Dammerman showed that the O'Neill's A-tensor vanishes as long as there is a fixed point.
\end{remark}
The integrability of horizontal distribution also holds if instead of requiring $g$ to be Einstein, we require it to be a Ricci soliton.
\begin{proposition}
    If $T^2$ acts by isometries on a simply connected closed  4-dimensional Ricci soliton, then the action is geometrically polar.
\end{proposition}
\begin{proof}
    By \cite{pe02}, a closed Ricci soliton must be a gradient soliton, i.e. it satisfies $\text{Ric}_g+\text{Hess}f=\lambda g$ for some potential function $f$. If $\text{Hess}f$ is invariant under the action of a compact connected Lie group, then $f$ must be invariant too, as it must be constant over the orbit of any closed one parameter subgroup. For $X,Y$ and $V,W$ as in previous paragraph,
    we have \[\begin{bmatrix}
        \text{Hess }f(X,V_1)& \text{Hess }f(X,V_2)\\
        \text{Hess }f(Y,V_1)& \text{Hess }f(Y,V_2)
    \end{bmatrix}=\begin{bmatrix}
        -\frac{1}{2\mu^2}\metric{[X,Y]}{V_1}\partial_y(f)&-\frac{1}{2\mu^2}\metric{[X,Y]}{V_2}\partial_y(f)\\
        \frac{1}{2}\metric{[X,Y]}{V_1}\partial_x(f)&\frac{1}{2}\metric{[X,Y]}{V_2}\partial_x(f)
    \end{bmatrix}\]
    Hence the Ricci curvatures give us the following conditions \begin{align*}
    \frac{\mu}{2e^{-f}\sqrt{\det(g^\Omega_{ij})}}\partial_x(\frac{\sqrt{\det(g^\Omega_{ij})}}{\mu}\metric{[X,Y]}{V_1}e^{-f})=0\\
\frac{1}{2e^{-f}\mu\sqrt{\det(g^\Omega_{ij})}}\partial_y(\frac{\sqrt{\det(g^\Omega_{ij})}}{\mu}\metric{[X,Y]}{V_1}e^{-f})=0\\
\frac{\mu}{2e^{-f}\sqrt{\det(g^\Omega_{ij})}}\partial_x(\frac{\sqrt{\det(g^\Omega_{ij})}}{\mu}\metric{[X,Y]}{V_2}e^{-f})=0\\
\frac{1}{2e^{-f}\mu\sqrt{\det(g^\Omega_{ij})}}\partial_y(\frac{\sqrt{\det(g^\Omega_{ij})}}{\mu}\metric{[X,Y]}{V_2}e^{-f})=0
\end{align*}
which then, by the same reason as in the Einstein case, implies that $\alpha,\beta$ are both constants over the principal stratum. Since $\alpha,\beta$ extends to continuous function that vanishes on the singular strata, they vanishes everywhere, hence the horizontal distribution is integrable.
\end{proof}
     



\section{\texorpdfstring{Dirichlet eigenfunctions as orbit volume}{Dirichlet eigenfunctions as orbit volume}}\label{4}
Now we restrict our attention to geometrically polar actions. In this case, the horizontal distribution is integrable, and the integral submanifolds give totally geodesic embeddings of $\Sigma$ into $M$. What we need are the following consequences of the Einstein equation. These results can be found in either \cite{Da06} or the author's thesis.\par
Let $B_X(V_i,V_j):=-\metric{\nabla_{X}V_i}{V_j}$ denote the second fundamental form of the orbit, and $H:=g^{ij}_\Omega\nabla_{V_i}V_j$ denote the mean curvature. We use $\phi$ to denote the orbit volume as a function on the orbit space. We shall use $\Delta_\Sigma, \text{Hess}_\Sigma$ to denote the Laplacian and Hessian operators induced by the metric on the orbit space. We have the following curvature expressions for the Ricci curvature and scalar curvature for a metric invariant under a geometrically polar $T^2$-action: 
\begin{align}
\Ric{V_i}{V_j}&=-\frac{1}{2}\Delta_{\Sigma}\metric{V_i}{V_j}+\frac{1}{2}\metric{H}{\nabla_\Sigma \metric{V_i}{V_j}}+2g^{st}_\Sigma \metric{(\nabla_{X_s}V_i)^\bot}{(\nabla_{X_t}V_j)^\bot}\\
\Ric{X}{Y}&=\text{Ric}_\Sigma(X,Y)-\phi^{-1}\text{Hess}_\Sigma\phi(X,Y)+\metric{H}{X}\metric{H}{Y}-\metric{B_X}{B_Y}\\
\Ric{V_i}{V_j}g^{ij}_\Omega&=-\phi^{-1}\Delta_{\Sigma}\phi\\
\Ric{X_s}{X_t}g^{st}_\Sigma&=\text{Scal}_\Sigma-\phi^{-1}\Delta_{\Sigma}\phi+|H|^2-|B|^2\\
\text{Scal}&=\text{Scal}_\Sigma-2\phi^{-1}\Delta_{\Sigma}\phi+|H|^2-|B|^2
\end{align}
It follows that by looking at the orbit-wise trace of the Ricci curvature and keeping in mind that $\phi$ is a non-negative function that vanishes only on the boundary of $\Sigma$:
\begin{proposition}
    Assume $\text{Ric}=\Lambda g$. Let $\phi$ be the orbit-volume function defined over the orbit space $\Sigma$. If $\Sigma$ is equipped with the quotient metric $g_{\Sigma}$, then we have $-\Delta_{\Sigma}\phi=2\Lambda\phi$. In other words, $\phi$ is a principal Dirichlet eigenfunction of the Laplace-Beltrami operator on $\Sigma$, and the principal eigenvalue is given by $2\Lambda$.  
\end{proposition}

\section{Uniqueness of \texorpdfstring{the Einstein metric when $\chi(M)=4$}{the Einstein metric when chi(M)=4}}\label{5}
We shall now investigate the restriction that the Einstein equations impose over $\partial\Sigma$, i.e. the union of singular strata. 
\begin{lemma}\label{5.1}
    For $(M^4,g)$ a simply connected closed Einstein 4-manifold with isometric $T^2$-action, if the orbit space $\Sigma$ has constant sectional curvature, then all its edges must have equal length.
\end{lemma}
\begin{proof}
Let $Z$ be the codimention one stratum with isotropy subgroup $S(1,0)$. As mentioned in section \ref{2}, we can lift a Fermi coordinate on $\Sigma$ to a tubular neighborhood of $Z$, and use exponential map to obtain a  chart $(r,\varphi, t,\theta)$. Since the action is geometrically polar, we can assume $c_{ij}$ all vanish. In these coordinates the metric takes the form
\[    ds^2=dr^2+\mu^2dt^2+a^2d\varphi^2+2bd\varphi d\theta+d^2d\theta^2\]
Using the conditions for smoothness in lemma \ref{2.1}, we can set 
\begin{align*}
    \mu^2&=1-\frac{1}{2}\text{sec}_\Sigma(0,t) r^2+O(r^4)\\
    a^2&=r^2+a_0(t)r^4+O(r^6)\\
    b&=b_0(t)r^2+O(r^4)\\
    d^2&=d_0(t)+d_2(t)r^2+O(r^4)
\end{align*}
We can then compute the Ricci curvature using these parameters. We get that when $r=0$, the only non-zero components of the Ricci tensor are the diagonal elements:
\begin{align*}
\Ric{\partial_r}{\partial_r}&=-3a_0-\frac{1}{d_0}(d_2-b_0^2)+\text{sec}_{\Sigma}\\
\Ric{\partial_t}{\partial_t}&=-\frac{\partial_y^2d_0}{2d_0}+\frac{(\partial_yd_0)^2}{4d_0^2}+2\text{sec}_{\Sigma}\\
\Ric{\partial_\varphi}{\partial_\varphi}&=-3a_0-\frac{1}{d_0}(d_2-b_0^2)+\text{sec}_{\Sigma}\\
\Ric{\partial_\theta}{\partial_\theta}&=-\frac{\partial_y^2d_0}{2d_0}+\frac{(\partial_yd_0)^2}{4d_0^2}-\frac{2}{d_0}(d_2-b_0^2)
\end{align*}
If the quotient of $Z$ has length $l$ in $\Sigma$, then the length $\sqrt{d_0(t)}$ of the $T^2$ orbits for these points inside $Z$ is determined by the ODE boundary value problem: we need $\sqrt{d_0(0)}=\sqrt{d_0(l)}=0$ and $-\frac{\partial_y^2\sqrt{d_0}}{\sqrt{d_0}}+2\text{sec}_{\Sigma}=\Lambda$. In fact, from the smoothness conditions near the fixed points we know $\frac{d^2}{dt^2}|_{t=0}d_0=-\frac{d^2}{dt^2}|_{t=l}d_0=2$, which grants us extra Neumann boundary conditions $\frac{d}{dt}|_{t=0}\sqrt{d_0}=-\frac{d}{dt}|_{t=l}\sqrt{d_0}=1$. If we further assume that the orbit space has constant sectional curvature, then we can solve the ODE explicitly and get \[\sqrt{d_0}(t)=(\Lambda-2\text{sec}_\Sigma)^{-\frac{1}{2}}\sin((\Lambda-2\text{sec}_\Sigma)^\frac{1}{2}t)\]
 In particular, we have $l(\Lambda-2\text{sec}_\Sigma)^{\frac{1}{2}}=\pi$. Thus if $\Sigma$ has constant sectional curvature, all its edges must have the same length $(\Lambda-2\text{sec}_\Sigma)^{-\frac{1}{2}}\pi$.
 \end{proof}
Combining with the Ricci curvature formulae in the previous section, we have the following proposition:
\begin{proposition}
If we fix a $T^2$-action on $M$ and a metric $(\Sigma,g_\Sigma)$ on the orbit space, then there exists at most one Einstein metric on $M^4$ invariant under the $T^2$-action with orbit space isometric to $g_\Sigma$.
\end{proposition}
\begin{proof}
    Given a metric over the orbit space, the formula for $\Ric{V_i}{V_j}$ in the previous section can be interpreted as that $g_{ij}^\Omega$ must satisfy a second order linear PDE along with prescribed Dirichlet and Neumann boundary data at the same time. Indeed, using the formula for the trace of Ricci curvature along the horizontal direction, we know the Einstein equations require that $\text{Scal}_\Sigma+|H|^2-|B|^2=0$. But we also know that for any symmetric matrix field on $\Sigma$ with adjugate $\text{adj}(M)$, there must be the identity $g^{st}_\Sigma X_s(M)\text{adj}(M)X_t(M)=g^{st}_\Sigma X_s(\det(M))X_t(M)-Mg^{st}_\Sigma X_s(\text{adj}(M))X_t(M)$ and $g^{st}_\Sigma X_s(\text{adj}(M))X_t(M)$ is always a scalar matrix. Therefore we know the following identity must hold:
    \begin{align*}
        \metric{H}{\nabla_\Sigma \metric{V_i}{V_j}}+2g^{st}_\Sigma \metric{(\nabla_{X_s}V_i)^\bot}{(\nabla_{X_t}V_j)^\bot}
        &=-\frac{1}{2\phi^2}\metric{\nabla_\Sigma\phi^2}{\nabla_\Sigma\metric{V_i}{V_j}}\\
        &+\frac{1}{2}g^{kl}_\Omega g^{st}_\Sigma X_s\metric{V_i}{V_k}X_t\metric{V_j}{V_l}\\
        &=(|B|^2-|H|^2)\metric{V_i}{V_j}
    \end{align*}
    Putting it back to the formula for $\Ric{V_i}{V_j}$, we get that all entries of $g_{ij}^\Omega$ must satisfy 
    \begin{align}\label{eq:9}
    -\frac{1}{2}\Delta_{\Sigma}u+\frac{1}{2\phi}\metric{\nabla_{\Sigma}\phi}{\nabla_{\Sigma} u}+2\sec_{\Sigma}u=\Lambda u
    \end{align}
    Noticing that $\phi$ is determined by the principal Dirichlet eigenfunction of the orbit space, we have thus obtained a second order linear elliptic PDE for $g_{ij}^\Omega$ over the interior of $\Sigma$.\par
    An obstacle is that the coefficients of the PDE blow up near $\partial\Sigma$. However, for our purpose we only need to consider a very special class of solutions. To be more specific, we use $R^{\text{even}}$ to denote the set of functions on $\Sigma$ that admits a smooth even extension across all edges, as was the convention in section \ref{2}.
    We need to show the following:
    \begin{lemma}\label{5.2}
        Consider the equation \ref{eq:9}. Given any Dirichlet boundary condition there is at most one solution to the boundary value problem among functions in $R^{\text{even}}$.
    \end{lemma}
    \begin{proof}
    Observe that $u_0:=g_{11}^\Omega+g_{22}^\Omega$ is a solution and it is positive except at vertices, where it is zero. By Lemma \ref{2.2}, it is an element in $R^{\text{even}}$.
    Now assume $u$ is a solution among $R^{\text{even}}$ to the PDE satisfying homogeneous Dirichlet boundary condition.
    Since by Lemma \ref{2.2} at each vertex the second derivatives of $u_0$ are non-vanishing, we know $w:=\frac{u}{u_0}$ continues up to the boundary, vanishes along the boundary, and we have in the interior of $\Sigma$
\begin{align*}
-\frac{1}{2}\text{div}_\Sigma(u_0^2\nabla_\Sigma w)+\frac{1}{2\phi}\metric{\nabla_{\Sigma}\phi}{u_0^2\nabla_\Sigma w}
&=-\frac{1}{2}(u_0\Delta_{\Sigma}u-u\Delta_{\Sigma}u_0)\\
&+\frac{1}{2\phi}\metric{\nabla_{\Sigma}\phi}{u_0\nabla_{\Sigma} u-u\nabla_{\Sigma}u_0}\\
&=0
\end{align*}
So we can apply the strong maximum principle (c.f. \cite{PS04}) to $w$ and conclude that it cannot attain maximum in the interior of $\Sigma$. Indeed, for any $\epsilon$-neighborhood $N_\epsilon$ of $\partial\Sigma$, we have $\max\limits_{x\in \Sigma-N_\epsilon}w(x)=\max\limits_{x\in \partial(\Sigma-N_\epsilon)}w(x)$. Let $\tilde{w}(\epsilon):=\max\limits_{x\in \Sigma-N_\epsilon}w(x)$, then $w(\epsilon)$ is a non-increasing function. For each $\epsilon$, pick a point $x_\epsilon\in \partial(\Sigma-N_\epsilon)$ such that $\tilde{w}(\epsilon)=w(x_\epsilon)$. If $\{\epsilon_n\}$ is a sequence approaching $\epsilon_0$, then by compactness of $\Sigma$ there is a converging subsequence of $x_{\epsilon_n}$, and by the continuity of distance function the limit point $x$ is in $\partial(\Sigma-N_{\epsilon_0})$. Since $w$ is continuous up to the boundary, we have $\limsup\limits_{n\to\infty}\tilde{w}(\epsilon_n)=\limsup\limits_{n\to\infty}w(x_{\epsilon_n})=w(x)\leq \tilde{w}(\epsilon_0)$, so $\tilde{w}$ is upper semicontinuous. Actually, with a bit more effort it is not hard to see it is continuous.
Now since $\tilde{w}(0)=0$, we get $\tilde{w}\leq 0$ for all $\epsilon\geq 0$. The same argument applied to $-w$ forces $w$ to be constantly zero, so $u$ is also constantly zero.
    \end{proof}
For $g_{ij}^\Omega$, they all extend to smooth even functions across the edges, and the Dirichlet boundary conditions are determined by the action as well as the ODEs for $d_0$ along the edges. We can without loss of generality let $V_1,V_2$ be the action fields of $S(1,0),S(0,1)$, which are the isotropy groups of two adjacent codimension-two strata $Z,Z'$. Since the degree of the map $S(1,0)\to T^2\to T^2/S(p,q)$ is $q$ and the degree of the map $S(0,1)\to T^2\to T^2/S(p,q)$ is $p$, we know along a codimension-two stratum with isotropy group $S(p,q)$ we must have $g_{ij}^\Omega=\begin{bmatrix}
    q^2&pq\\
    pq&p^2
\end{bmatrix}\sqrt{d_0}$ where the length $\sqrt{d_0}$ of $T^2$ orbits are determined via solving the ODE. Thus in the interior of $\Sigma$ they are uniquely determined as well.
\end{proof}
\begin{remark}
    The formula for $\Ric{X}{Y}$ in the previous section also has a geometric interpretation. It is saying that $g_{ij}^\Omega$ must satisfy an ODE along every geodesic in $\Sigma$. Namely, if we parametrize a geodesic $\gamma(s)$ to have unit speed, we have
    \begin{align*}
        \text{sec}_{\Sigma}-(\det  g_{ij}^\Omega)^{-\frac{1}{2}}\frac{d^2}{ds^2}(\sqrt{\det g_{ij}^\Omega})+\frac{1}{2\det(g_{ij}^\Omega)}\det( \frac{d}{ds}g_{ij}^\Omega)=\Lambda
    \end{align*}
\end{remark}
We can now prove that among all actions over $M=S^2\times S^2$ or $\mathbb{CP}^2\#\bar{\mathbb{CP}}^2$ there is only one of them admitting invariant Einstein metric of non-negative sectional curvature. In this case the orbit space has four edges, and without loss of generalizty we can assume the isotropy groups are given consecutively by $S(1,0)$, $S(0,1)$, $S(1,0)$, $S(p,1)$, where $p$ is an even integer if the topology is $S^2\times S^2$, and an odd integer if the topology is $\mathbb{CP}^2\#\bar{\mathbb{CP}}^2$ \cite{Hi51}. Since the sectional curvature is non-negative and $\Sigma$ can be embedded into $M$ as a totally geodesic submanifold, we know the orbit space must be a flat square. 
\begin{proposition}
If $(\Sigma,g_\Sigma)$ is a flat square and $M$ has a $T^2$-invariant Einstein metric, then $M$ is isometric to the product $S^2\times S^2$ of two spheres of the same radius, and the action must be the product action.
\end{proposition}
\begin{proof}
    We can solve for all entries of the metric tensor explicitly. Let the square be given by $\{(x,y)|0\leq x,y\leq \pi\}$, and let the codimension-two stratum corresponding to $y=0,\pi$ have isotropy groups $S(0,1), S(p,1)$. The orbit volume function is given by the principal eigenfunction $\phi(x,y)=\sin(x)\sin(y)$, as one can verify its first derivatives along a codimension-two stratum. We then observe that
    \begin{align*}
        g_{11}^\Omega(x,y)&=\sin^2(x)\\
        g_{22}^\Omega(x,y)&=\frac{p^2}{2}\sin^2(x)(1-\cos(y))+\sin^2(y)\\
        g_{12}^\Omega(x,y)&=\frac{p}{2}\sin^2(x)(1-\cos(y))
    \end{align*}
    give the unique solutions to the second order linear PDEs mentioned in this section. However, they do not yield Einsteim metrics unless $p=0$: we have $\det(g_{ij}^\Omega)=\sin^2(x)\sin^2(y)+\frac{p^2}{2}\sin^4(x)\sin^2(y)\neq \phi^2$ unless $p=0$. 
\end{proof}

\section{Uniqueness of \texorpdfstring{the Einstein metric when $\chi(M)=2$}{the Einstein metric when chi(M)=2}}\label{6}
It remains to show that in the case $M=S^4$, a $T^2$-invariant Einstein metric is isometric to the round sphere. By \cite{OR70}, there is only one $T^2$-action on $S^4$ and its orbit space has only two edges. The two codimension-two strata have isotropy groups $S(1,0),S(0,1)$ respectively. Such an action is an example of a diagonal action, which we define below:
\begin{definition}\label{def:1}
    A $T^2$-action on a four-manifold $M$ is called a diagonal action if all the isotropy groups of codimension-two strata are either $S(1,0)$ or $S(0,1)$. A metric invariant under such $T^2$-action is called a diagonal metric if the action is geometrically polar and the action fields of $S(1,0)$ and $S(0,1)$ are orthogonal to each other everywhere.
\end{definition}
The motivation for studying diagonal metrics is provided by the following result:
\begin{proposition}
If there exists an Einstein metric on a closed simply connected four manifold that is invariant under a diagonal action, then it must be a diagonal metric.
\end{proposition}
\begin{proof}
    We let $V_1,V_2$ denote the action fields generated by $S(1,0),S(0,1)$ respectively, then $g_{12}^\Omega$ satisfies the linear PDE \ref{eq:9} in the previous section, and it vanishes along every codimension-two stratum. Therefore Lemma \ref{5.2} forces it to vanish everywhere, i.e. the action fields are everywhere orthogonal to each other.
\end{proof}
For diagonal metrics, we use $2\pi a,2\pi b$ to denote the lengths of the orbits of $S(1,0),S(0,1)$ at each point. Let $V,W$ be the corresponding action fields with length $a,b$. The curvature formulae are then greatly simplified. In fact, the only non-zero components are:
\begin{align*}
\metric{R(X,Y)X}{Y}&=\text{sec}_\Sigma|X\wedge Y|^2&
\metric{R(X,V)X}{V}&=-a\text{Hess}_\Sigma a(X,X)\\
\metric{R(X,V)Y}{V}&=-a\text{Hess}_\Sigma a(X,Y)&
\metric{R(X,W)X}{W}&=-b\text{Hess}_\Sigma b(X,X)\\
\metric{R(X,W)Y}{W}&=-b\text{Hess}_\Sigma b(X,Y)&
\metric{R(Y,V)Y}{V}&=-a\text{Hess}_\Sigma a(Y,Y)\\
\metric{R(Y,W)Y}{W}&=-b\text{Hess}_\Sigma b(Y,Y)&
\metric{R(V,W)V}{W}&=-ab\metric{\nabla_\Sigma a}{\nabla_\Sigma b}
\end{align*}


Let $s:=\text{sec}_\Sigma, p:=-a^{-1}b^{-1}\metric{\nabla_\Sigma a}{\nabla_\Sigma b}, A:=-a^{-1}\text{Hess}_\Sigma a, B:=-b^{-1}\text{Hess}_\Sigma b$. If a diagonal metric has non-negative sectional curvature, then $s,p,A,B$ are all non-negative. The curvature operator in the basis $\{\frac{1}{2}(\hat{X}\wedge \hat{Y}+\hat{V}\wedge\hat{W}), \frac{1}{2}(\hat{X}\wedge \hat{V}-\hat{Y}\wedge\hat{W}), \frac{1}{2}(\hat{X}\wedge \hat{W}+\hat{Y}\wedge\hat{V}), \frac{1}{2}(\hat{X}\wedge \hat{Y}-\hat{V}\wedge\hat{W}), \frac{1}{2}(\hat{X}\wedge \hat{V}+\hat{Y}\wedge\hat{W}), \frac{1}{2}(\hat{X}\wedge \hat{W}-\hat{Y}\wedge\hat{V})\}$ is given by 
\[\begin{bmatrix}
\frac{1}{2}(s+p)&0&0&\frac{1}{2}(s-p)&0&0\\
0&\frac{1}{2}(A_{11}+B_{22})&\frac{1}{2}(A_{12}-B_{12})&0&\frac{1}{2}(A_{11}-B_{22})&\frac{1}{2}(A_{12}+B_{12})\\
0&\frac{1}{2}(A_{12}-B_{12})&\frac{1}{2}(A_{22}+B_{11})&0&\frac{1}{2}(A_{12}+B_{12})&\frac{1}{2}(B_{11}-A_{22})\\
\frac{1}{2}(s-p)&0&0&\frac{1}{2}(s+p)&0&0\\
0&\frac{1}{2}(A_{11}-B_{22})&\frac{1}{2}(A_{12}+B_{12})&0&\frac{1}{2}(A_{11}+B_{22})&\frac{1}{2}(A_{12}-B_{12})\\
0&\frac{1}{2}(A_{12}+B_{12})&\frac{1}{2}(B_{11}-A_{22})&0&\frac{1}{2}(A_{12}-B_{12})&\frac{1}{2}(B_{11}+A_{22})
\end{bmatrix}\]
Therefore in the non-negative sectional curvature case the upper-left and lower-right three-by-three blocks are both non-negative. By lemma 2.1 of \cite{Ha97}, a metric has non-negative isotropic curvature if and only if the sum of the two smallest eigenvalues of these two blocks are non-negative. Therefore we have proven:

\begin{theorem}\label{6.2}
A diagonal metric on $M^4$ with non-negative sectional curvature has non-negative isotropic curvature.
\end{theorem}
On the other hand, by \cite{Br10} we know an Einstein metric with non-negative isotropic curvature must be locally symmetric. Thus a $T^2$-invariant Einstein metric on $S^4$ with non-negative sectional curvature is isometric to the round metric.\par

\section{Other closed orientable examples}\label{7}
In this section we prove the following corollary of our main theorem, extending the classification to orientable examples.
\begin{theorem}
    Assume $(M^4,g)$ is an orientable, non-negatively curved, closed Einstein manifold with $T^2$-symmetry, then it is isometric to one of the following:
    \begin{enumerate}
        \item A simply connected example mentioned in \ref{1.1}
        \item One of the two isometric $\Z_2$ quotients of $S^2\times S^2$
        \item A flat 4-torus
        \item A hyperelliptic K\"ahler surface. These manifolds are the quotients of $T^4$ by some finite group action, and fall into seven diffeomorphism classes.
    \end{enumerate}
\end{theorem}
\begin{proof}
    We first deal with the $\Lambda>0$ case. In this case, the universal cover is given by one of the metrics mentioned in \ref{1.1}. The Lefschetz fixed point theorem or Synge's theorem guarantees that any orientation-preserving isometry of $S^4$ or $\mathbb{CP}^2$ must have fixed points. 
    For the Einstein metric on $S^2\times S^2$, the isometry group is a semidirect product $G:=(O(3)\times O(3))\rtimes \Z_2$, with $\Z_2$ permuting the two factors. There are only two conjugacy classes of non-trivial subgroups of $G$ acting on $S^2\times S^2$ without fixed point, both isomorphic to $\Z_2$. If we let $A$ denote the antipodal map on $S^2$ and $T$ denote a reflection across an equator, then up to conjugacy one of the two subgroups is generated by the involution $(A,A)$, and the other is generated by $(A,T)$.
    Finally, as mentioned in \cite{Le21}, since $S^2\times S^2/(A,T)$ is spin and $S^2\times S^2/(A,A)$ is not, they cannot be isometric to each other.\par
    We then deal with the $\Lambda=0$ case. In this case, the manifold $M$ is Ricci-flat with non-negative sectional curvature, hence it is flat. By the Weitzenb\"ock formula, 
    the Killing fields on a closed Ricci-flat manifold are parallel (c.f. \cite{Es94} Theorem 13.1.b).
    Let $\alpha_1,\alpha_2$ be the one forms dual to the Killing fields, then $\omega:=\alpha_1\wedge \alpha_2+*(\alpha_1\wedge \alpha_2)$ is a parallel symplectic form.
    The almost complex structure $J$ induced by $\omega$ is therefore also parallel, so $M$ is a K\"ahler surface. The flat K\"ahler surfaces have been classified in \cite{DHS09}, and they are either a flat torus or one of the seven hyperelliptic surfaces. Since the hyperelliptic surfaces all have $b_2=2$, they admit two independent parallel vector fields, and therefore these hyperelliptic surfaces all admit isometric $T^2$-actions.
\end{proof}

\section{Invariance under the Ricci flow}\label{8}
In this section we show that geometrically polar and diagonal are properties preserved by the Ricci flow.
\begin{theorem}
If $(g_t)_{t\in[0,T)}$ is a family of smooth metric on a simply connected closed 4-manifold $M$ satisfying the Ricci flow equation $\partial_tg_t=-2\text{Ric}(g_t)$. If there exists an isometric geometrically polar $T^2$-action on $(M,g_0)$, then the isometric $T^2$-action on $(M,g_t)$ is geometrically polar for all $t$.
\end{theorem}
\begin{proof}
Let $\bar{X},\bar{Y}$ be two vector fields on the orbit space, and $X_t,Y_t$ be the horizontal lift of them with respect to $g_t$. For any pair of $\bar{X},\bar{Y}$ linearly independent at $p$, the quantity $\frac{\metric{[X_t,Y_t]^\bot}{[X_t,Y_t]^\bot}}{|X_t\wedge Y_t|^2}$ assumes the same value, denoted as $\omega_t(p)$. The strategy is to pull it pack to $M$ and show that if the integral $\int_{M}\pi^*\omega_t d\text{vol}(g_t)$ vanishes at $t=0$, then it vanishes for all time. From previous sections we know the integral is bounded.\par
Let $\theta_1,\theta_2\in[0,2\pi)$ parametrize $T^2$, and $V_1,V_2$ be the action fields generated by $\partial_{\theta_1},\partial_{\theta_2}$. We use $\metric{-}{-}$ to denote the metric at time $t$, and $\metric{-}{-}_\Sigma$ to denote the metric on $\Sigma$ at time $t$. As before $g_{ij}^\Omega:=\metric{V_i}{V_j}$ denote the metric along the orbit. To abbreviate notation, we define $\phi:=\sqrt{\det(g^\Sigma_{ij}})$, $\mu:=|X_t\wedge Y_t|^2$. 
\begin{lemma}
    For $i=1,2$, let $\alpha^i:=\frac{1}{\mu}g^{ij}_\Omega\metric{[X_t,Y_t]}{V_j}$, then $\alpha^i$ evolves according to the following formula under the Ricci flow:
\begin{align*}
    \partial_t(\alpha^i)=\text{div}_\Sigma\frac{1}{\phi}g^{ij}_\Omega\nabla_\Sigma(\phi g_{jk}^\Omega\alpha^k)-\frac{\partial_t\mu}{\mu}\alpha^i
    \end{align*}
\end{lemma}
\begin{proof}

We have $X_t=X_0-g^{ij}_\Omega\metric{X_0}{V_i}V_j$, so \begin{align*}
    \partial_t\metric{X_0}{V_k}=-2\Ric{X_t}{V_k}-2g^{ij}_\Omega\metric{X_0}{V_i}\Ric{V_j}{V_k}
    \end{align*}
    We also have a formula for the components of O'Neill's A-tensor as follow: 
    \begin{align*}
    \metric{[X_t,Y_t]}{V_k}&=\metric{[X_0,Y_0]}{V_k}+[Y_0(g^{ij}_\Omega\metric{X_0}{V_i})-X_0(g^{ij}_\Omega)\metric{Y_0}{V_i})]\metric{V_j}{V_k}\\
    &=\metric{[X_0,Y_0]}{V_k}+Y_0(\metric{X_0}{V_k})-X_0(\metric{Y_0}{V_k})\\
    &-g^{ij}_\Omega\metric{X_0}{V_i}Y_0(\metric{V_j}{V_k})+g^{ij}_\Omega\metric{Y_0}{V_i}X_0(\metric{V_j}{V_k})
    \end{align*}

    Taking time derivative of the quantity above, we get the following evolution formula:
    \begin{align*}
    \partial_t\metric{[X_t,Y_t]}{V_k}&=-2\Ric{[X_0,Y_0]}{V_k}
    -2Y_0(\Ric{X_0}{V_k})+2X_0(\Ric{Y_0}{V_k})\\
    &+2g^{ij}_\Omega\metric{X_0}{V_i}Y_0(\Ric{V_j}{V_k})-2g^{ij}_\Omega\metric{Y_0}{V_i}X_0(\Ric{V_j}{V_k})\\
    &+2g^{ij}_\Omega\Ric{X_0}{V_i}Y_0(\metric{V_j}{V_k})-2g^{ij}_\Omega\Ric{Y_0}{V_i}X_0(\metric{V_j}{V_k})\\
    &-2g^{im}_\Omega g^{jl}_\Omega\Ric{V_l}{V_m}\metric{X_0}{V_i}Y_0(\metric{V_j}{V_k})\\
    &+2g^{im}_\Omega g^{jl}_\Omega\Ric{V_l}{V_m} \metric{Y_0}{V_i}X_0(\metric{V_j}{V_k})
\end{align*}
However, we still need to simplify this expression. We shall do it term by term. First we know that the second and third term can be expanded into the following:
\begin{align*}
    &-2Y_0(\Ric{X_0}{V_k})+2X_0(\Ric{Y_0}{V_k})\\
    &=-2Y_0(\Ric{X_t}{V_k})-2g^{ij}_\Omega \metric{X_0}{V_i}Y_0(\Ric{V_j}{V_k})\\
    &-2g^{ij}_\Omega Y_0(\metric{X_0}{V_i})\Ric{V_j}{V_k}
    +2g^{im}_\Omega g^{jl}_\Omega Y_0(\metric{V_m}{V_l})\metric{X_0}{V_i}\Ric{V_j}{V_k}\\
    &+2X_0(\Ric{Y_t}{V_k})+2g^{ij}_\Omega \metric{Y_0}{V_i}X_0(\Ric{V_j}{V_k})
    +2g^{ij}_\Omega X_0(\metric{Y_0}{V_i})\Ric{V_j}{V_k}\\
    &-2g^{im}_\Omega g^{jl}_\Omega X_0(\metric{V_m}{V_l})\metric{Y_0}{V_i}\Ric{V_j}{V_k}
\end{align*}
Put it back into the previous equation, we get 
\begin{align*}
    \partial_t\metric{[X_t,Y_t]}{V_k}&=-2\Ric{[X_0,Y_0]}{V_k}
    -2Y_0(\Ric{X_t}{V_k})+2X_0(\Ric{Y_t}{V_k})\\
    &-2g^{ij}_\Omega Y_0(\metric{X_0}{V_i})\Ric{V_j}{V_k}
    +2g^{im}_\Omega g^{jl}_\Omega Y_0(\metric{V_m}{V_l})\metric{X_0}{V_i}\Ric{V_j}{V_k}\\
    &+2g^{ij}_\Omega X_0(\metric{Y_0}{V_i})\Ric{V_j}{V_k}
    -2g^{im}_\Omega g^{jl}_\Omega X_0(\metric{V_m}{V_l})\metric{Y_0}{V_i}\Ric{V_j}{V_k}\\
    &+2g^{ij}_\Omega\Ric{X_0}{V_i}Y_0(\metric{V_j}{V_k})-2g^{ij}_\Omega\Ric{Y_0}{V_i}X_0(\metric{V_j}{V_k})\\
    &-2g^{im}_\Omega g^{jl}_\Omega\Ric{V_l}{V_m}\metric{X_0}{V_i}Y_0(\metric{V_j}{V_k})\\
    &+2g^{im}_\Omega g^{jl}_\Omega\Ric{V_l}{V_m} \metric{Y_0}{V_i}X_0(\metric{V_j}{V_k})
    \end{align*}
    We can use the following identity to simplify the last four terms.
    \begin{align*}
    &2g^{ij}_\Omega\Ric{X_t}{V_i}Y_0(\metric{V_j}{V_k})-2g^{ij}_\Omega\Ric{Y_t}{V_i}X_0(\metric{V_j}{V_k})\\
    &=2g^{ij}_\Omega\Ric{X_0}{V_i}Y_0(\metric{V_j}{V_k})-2g^{ij}_\Omega\Ric{Y_0}{V_i}X_0(\metric{V_j}{V_k})\\
    &-2g^{im}_\Omega g^{jl}_\Omega\Ric{V_l}{V_m}\metric{X_0}{V_i}Y_0(\metric{V_j}{V_k})\\
    &+2g^{im}_\Omega g^{jl}_\Omega\Ric{V_l}{V_m} \metric{Y_0}{V_i}X_0(\metric{V_j}{V_k})
    \end{align*}
We can then reduce the evolution equation for the components of O'Neill's A-tensor:
\begin{align*}
    \partial_t\metric{[X_t,Y_t]}{V_k}&=-2\Ric{[X_0,Y_0]}{V_k}
    -2Y_0(\Ric{X_t}{V_k})+2X_0(\Ric{Y_t}{V_k})\\
    &+2g^{ij}_\Omega\Ric{X_t}{V_i}Y_0(\metric{V_j}{V_k})-2g^{ij}_\Omega\Ric{Y_t}{V_i}X_0(\metric{V_j}{V_k})\\
    &-2g^{ij}_\Omega Y_0(\metric{X_0}{V_i})\Ric{V_j}{V_k}+2g^{ij}_\Omega g^{ml}_\Omega Y_0(\metric{V_m}{V_i})\metric{X_0}{V_l}\Ric{V_j}{V_k}\\
    &+2g^{ij}_\Omega X_0(\metric{Y_0}{V_i})\Ric{V_j}{V_k}-2g^{ij}_\Omega g^{ml}_\Omega X_0(\metric{V_m}{V_i})\metric{Y_0}{V_l}\Ric{V_j}{V_k}\\
    &=-2\Ric{[X_0,Y_0]}{V_k}
    -2Y_0(\Ric{X_t}{V_k})+2X_0(\Ric{Y_t}{V_k})\\
    &+2g^{ij}_\Omega\Ric{X_t}{V_i}Y_0(\metric{V_j}{V_k})-2g^{ij}_\Omega\Ric{Y_t}{V_i}X_0(\metric{V_j}{V_k})\\
    &-2g^{ij}_\Omega\metric{[X_t,Y_t]-[X_0,Y_0]}{V_i}\Ric{V_j}{V_k}\\
    &=-2Y_0(\Ric{X_t}{V_k})+2X_0(\Ric{Y_t}{V_k})+2g^{ij}_\Omega\Ric{X_t}{V_i}Y_0(\metric{V_j}{V_k})\\
    &-2g^{ij}_\Omega\Ric{Y_t}{V_i}X_0(\metric{V_j}{V_k})-2g^{ij}_\Omega\metric{[X_t,Y_t]}{V_i}\Ric{V_j}{V_k}
    \end{align*}
    In other words, we get \[\partial_t(g^{ij}_\Omega\metric{[X_t,Y_t]}{V_j})=-2Y_0(g^{ij}_\Omega\Ric{X_t}{V_j})+2X_0(g^{ij}_\Omega\Ric{Y_t}{V_j})\]
    
    Now we can use the formulae for Ricci curvature obtained in the proof of theorem \ref{3.2}. The metric $g_{ij}^\Sigma$ determines a complex structure $J$ on $\Sigma$ such that
    \begin{align*}
        J\bar{X}&=-\frac{1}{|\bar{X}\wedge \bar{Y}|^2}(\metric{\bar{X}}{\bar{Y}}\bar{X}-\metric{\bar{X}}{\bar{X}}\bar{Y})\\
        J\bar{Y}&=-\frac{1}{|\bar{X}\wedge \bar{Y}|^2}(\metric{\bar{Y}}{\bar{Y}}\bar{X}-\metric{\bar{X}}{\bar{Y}}\bar{Y})
    \end{align*}
    The Ricci curvature formulae then give us the following: \begin{align*}
    \partial_t(g^{ij}_\Omega\metric{[X_t,Y_t]}{V_j})&=
    -2Y_0(g^{ij}_\Omega\Ric{X_t}{V_j})+2X_0(g^{ij}_\Omega\Ric{Y_t}{V_j})\\
    &=-\bar{X}(g^{ij}_\Omega\frac{\mu}{\phi}J\bar{Y}(\frac{\phi}{\mu}\metric{[X_t,Y_t]}{V_j}))
    +\bar{Y}(g^{ij}_\Omega\frac{\mu}{\phi}J\bar{X}(\frac{\phi}{\mu}\metric{[X_t,Y_t]}{V_j}))
    \\
    &=\mu\text{div}_\Sigma\frac{1}{\phi}g^{ij}_\Omega\nabla_\Sigma(\frac{\phi}{\mu}\metric{[X_t,Y_t]}{V_j})
    \end{align*}
    The evolution formula for $\alpha^i$ now follows by applying the Leibniz rule.
\end{proof}

Having obtained the evolution equation for the integrand, we then investigate the evolution of the integral. Based on the discussion of horizontal lifts near singular strata in section \ref{2}, we know $\partial_{\eta}(\phi g_{ij}^\Omega\alpha^i)\alpha^j=0$ along $\partial\Sigma$, where $\partial_\eta$ denote the derivative in unit normal direction. Through integration by parts, we get
    \begin{align*}
\partial_t\int_{M}\pi^*\omega_td\text{vol}(g_t)
&=4\pi^2\partial_t\int_{\Sigma}\phi g_{ij}^\Omega\alpha^i\alpha^jd\text{vol}(g_\Sigma)\\
&=4\pi^2\int_{\Sigma}[\partial_t(\phi g_{ij}^\Omega)\alpha^i\alpha^j+2\phi g_{ij}^\Omega\partial_t(\alpha^i)\alpha^j+\frac{\partial_t\mu}{\mu}\phi g_{ij}^\Omega\alpha^i\alpha^j]d\text{vol}(g_\Sigma)\\
&=4\pi^2\int_{\Sigma}[\partial_t(\phi g_{ij}^\Omega)\alpha^i\alpha^j+2\phi g_{ij}^\Omega\alpha^j\text{div}_\Sigma\frac{1}{\phi}g^{il}_\Omega\nabla_\Sigma(\phi g_{lk}^\Omega\alpha^k)\\
&-\frac{\partial_t\mu}{\mu}\phi g_{ij}^\Omega\alpha^i\alpha^j]d\text{vol}(g_\Sigma)\\
&=8\pi^2\int_{\partial\Sigma}\partial_{\eta}(\phi g_{ij}^\Omega\alpha^i)\alpha^jdl|_{\partial\Sigma}-\frac{\partial_t\mu}{\mu}\phi g_{ij}^\Omega\alpha^i\alpha^j]d\text{vol}(g_\Sigma)\\
&+4\pi^2\int_{\Sigma}[\partial_t(\phi g_{ij}^\Omega)\alpha^i\alpha^j
-\frac{2}{\phi}g^{il}_\Omega\metric{\nabla_\Sigma(\phi g_{ij}^\Omega\alpha^j)}{\nabla_\Sigma(\phi g_{lk}^\Omega\alpha^k)}_\Sigma\\
&=-8\pi^2\int_{\Sigma}\frac{1}{\phi}g^{il}_\Omega\metric{\nabla_\Sigma(\phi g_{ij}^\Omega\alpha^j)}{\nabla_\Sigma(\phi g_{lk}^\Omega\alpha^k)}_\Sigma d\text{vol}(g_\Sigma)\\
&+4\pi^2\int_{\Sigma}[\partial_t(\phi g_{ij}^\Omega)\alpha^i\alpha^j
-\frac{\partial_t\mu}{\mu}\phi g_{ij}^\Omega\alpha^i\alpha^j]d\text{vol}(g_\Sigma)
\end{align*}
 We notice that the first term is non-negative. Given any compact time interval $[0,T]$ on which the Ricci flow is non-singular, let $C$ be a constant bigger than $\max\{\max_M|\text{Ric}(g_t)||t\in [0,T]\}$ and $\max\{|\frac{\partial_t\mu}{\mu}|(p)|(p,t)\in M\times [0,T]\}$. By Jacobi's identity, $\phi^{-1}\partial_\phi=\frac{1}{2}g^{ij}_\Omega \partial_tg_{ij}^\Omega=-g^{ij}_\Omega \Ric{V_i}{V_j}\leq C$. We then have the following monotonicity:
\begin{align*}
e^{3Ct}\partial_t(e^{-3Ct}\int_{M}\pi^*\omega_td\text{vol}(g_t))
&=\partial_t\int_{M}\pi^*\omega_td\text{vol}(g_t)
-3C\int_{M}\pi^*\omega_td\text{vol}(g_t)\\ 
&\leq4\pi^2\int_{\Sigma}[\partial_t(\phi g_{ij}^\Omega)\alpha^i\alpha^j
-\frac{\partial_t\mu}{\mu}\phi g_{ij}^\Omega\alpha^i\alpha^j-3C\phi g_{ij}^\Omega\alpha^i\alpha^j]d\text{vol}(g_\Sigma)\\
&\leq 0
\end{align*}

Therefore if initially the action is geometrically polar, we have $\int_{M}\pi^*\omega_td\text{vol}(g_t))=0$ for all $t$, i.e. the action is geometrically polar for all time.
\end{proof}
A similar argument provides us with the following:
\begin{proposition}
Diagonal metrics are preserved under the Ricci flow.
\end{proposition}
\begin{proof}
    Based on the computation near singular strata, we know $\frac{(g_{12}^\Omega)^2}{\phi}$ is bounded over $\Sigma$, as $g_{12}^\Omega$ vanishes to second order. We can consider the evolution of the quantity $\int_{\Sigma}\phi^{-1}(g_{12}^\Omega)^2d\text{vol}_{\Sigma}$. From the computation in section \ref{4} we get the following:
    \begin{align*}
        \partial_t\phi=\frac{1}{2}\phi g^{ij}_\Omega\partial_tg_{ij}^\Omega=-\phi g^{ij}_\Omega\Ric{V_i}{V_j}=\Delta_\Sigma\phi
    \end{align*}
    Therefore we have the following evolution equation for our integrand:
    \begin{align*}
        \partial_t(\frac{(g_{12}^\Omega)^2}{\phi})&=\frac{2g_{12}^\Omega\partial_t(g_{12}^\Omega)}{\phi}-\frac{(g_{12}^\Omega)^2\partial_t\phi}
        {\phi^2}\\
        &=-\frac{4g_{12}^\Omega\Ric{V_1}{V_2}}{\phi}-\frac{(g_{12}^\Omega)^2\Delta_\Sigma\phi}
        {\phi^2}\\
        &=2g_{12}^\Omega\text{div}_\Sigma(\phi^{-1}\nabla_\Sigma g_{12}^\Omega)-4\phi^{-1}(|B|^2-|H|^2)(g_{12}^\Omega)^2-\frac{(g_{12}^\Omega)^2\Delta_\Sigma\phi}{\phi^2}
    \end{align*}
    Let $\mu$ denote the area element of $\Sigma$ as before. We then differentiate the integral over $\Sigma$ to get:
    \begin{align*}
        \partial_t\int_\Sigma \frac{(g_{12}^\Omega)^2}{\phi}d\text{vol}(g_\Sigma)&=2\int_\Sigma g_{12}^\Omega\text{div}_\Sigma(\phi^{-1}\nabla_\Sigma g_{12}^\Omega)d\text{vol}(g_\Sigma)\\
        &-4\int_\Sigma\phi^{-1}(|B|^2-|H|^2+\frac{1}{4}\phi^{-1}\Delta_\Sigma\phi-\frac{1}{4}\mu^{-1}\partial_t\mu)(g_{12}^\Omega)^2d\text{vol}(g_\Sigma)\\
        &=2\int_\Sigma g_{12}^\Omega\phi^{-1}\partial_\eta g_{12}^\Omega dl|_{\partial\Sigma}-2\int_\Sigma \phi^{-1}\metric{\nabla_\Sigma g_{12}^\Omega}{\nabla_\Sigma g_{12}^\Omega} d\text{vol}(g_\Sigma)\\
        &-4\int_\Sigma\phi^{-1}(|B|^2-|H|^2+\frac{1}{4}\phi^{-1}\Delta_\Sigma\phi-\frac{1}{4}\mu^{-1}\partial_t\mu)(g_{12}^\Omega)^2d\text{vol}(g_\Sigma)\\
        &=-2\int_\Sigma \phi^{-1}\metric{\nabla_\Sigma g_{12}^\Omega}{\nabla_\Sigma g_{12}^\Omega} d\text{vol}(g_\Sigma)\\
        &-\int_\Sigma\phi^{-1}(5g_{ij}^\Omega\Ric{V_i}{V_j}-4g_{st}^\Sigma\Ric{X_s}{X_t}+4\text{Scal}_\Sigma\\
        &-\mu^{-1}\partial_t\mu)(g_{12}^\Omega)^2d\text{vol}(g_\Sigma)
    \end{align*}
    The boundary term must vanish as $g_{12}^\Omega$ vanishes to second order, and $\phi$ only vanishes to first order. For any time interval $[0,T]$  on which the Ricci flow is smooth, let $C$ be a constant bigger than 
    $\max\{-(5g_{ij}^\Omega\Ric{V_i}{V_j}-4g_{st}^\Sigma\Ric{X_s}{X_t}+4\text{Scal}_\Sigma-\mu^{-1}\partial_t\mu)|x\in\Sigma,t\in[0,T]\}$, we get 
    \begin{align*}
        \partial_t\int_\Sigma \frac{(g_{12}^\Omega)^2}{\phi}d\text{vol}(g_\Sigma)\leq -2\int_\Sigma \phi^{-1}\metric{\nabla_\Sigma g_{12}^\Omega}{\nabla_\Sigma g_{12}^\Omega} d\text{vol}(g_\Sigma)
        +C\int_\Sigma\phi^{-1}(g_{12}^\Omega)^2d\text{vol}(g_\Sigma)
    \end{align*}
    Hence we have $e^{Ct}\partial_t(e^{-Ct}\int_\Sigma \frac{(g_{12}^\Omega)^2}{\phi}d\text{vol}(g_\Sigma))\leq 0$, so if the initial metric is diagonal we have $\int_\Sigma \frac{(g_{12}^\Omega)^2}{\phi}d\text{vol}(g_\Sigma)=0$ for all time, which implies that the metric is diagonal for all time.
\end{proof}

\section*{Declarations}
\begin{itemize}
     \item Ethics approval and consent to participate: Not Applicable.
    \item Funding: No funding was received to assist with the preparation of this manuscript.
    \item The author declares no conflict of interest.
\end{itemize}

\bibliography{research}

\end{document}